\documentclass[reqno]{amsart}
\usepackage{amssymb,amsthm,amsfonts,amstext}
\usepackage{amsmath}
\usepackage{graphicx}
\usepackage{eucal}
\usepackage[]{float}
\usepackage[]{epsfig}
\usepackage[]{pstricks}

\usepackage{mathptmx}       
\usepackage{helvet}         
\usepackage{courier}        
\usepackage{flafter}
\usepackage{amscd,amsthm}
\usepackage{amsfonts}
\usepackage{graphicx}
\usepackage{verbatim}
\usepackage[ansinew]{inputenc}
\makeatletter \@addtoreset{equation}{section} \makeatother

\renewcommand\thetable{\thesection.\@arabic\c@table}
\newtheorem{theorem}{Theorem}[section]

\newtheorem{corollary}[theorem]{Corollary}

\newtheorem{remark}{Remark}[section]

   \usepackage{wasysym}

\begin{document}

\author{Patr\'{\i}cia Gon\c{c}alves}
\address{ PUC-RIO, Departamento de Matem\'atica, Rua Marqu\^es de S\~ao Vicente, no. 225, 22453-900, Gávea, Rio de Janeiro, Brazil \\and\\
CMAT, Centro de Matem\'atica da Universidade do Minho, Campus de Gualtar, 4710-057 Braga, Portugal.}
\email{patricia@mat.puc-rio.br}

\title[ On the asymmetric zero-range in the rarefaction fan]{On the asymmetric zero-range in the rarefaction fan}
\date{\today}
\renewcommand{\subjclassname}{\textup{2000} Mathematics Subject Classification}


\begin{abstract}
 We consider one-dimensional asymmetric zero-range processes starting from a step decreasing profile leading, in the hydrodynamic limit, to the rarefaction fan of the associated hydrodynamic equation. Under that initial condition, and for {\em{ totally asymmetric jumps}}, we show that the weighted sum of joint
probabilities for second class particles sharing the same site is convergent and we compute its limit.
For {\em{partially asymmetric jumps}}, we derive the Law of Large Numbers for
 a second class particle, under the initial configuration in which all positive sites are empty, all negative
sites are occupied with infinitely many first class particles and there is a single second class particle at the origin. Moreover, we prove that among the infinite characteristics emanating from the position of the second class particle it picks randomly one of them. The randomness is given in terms of the weak solution of the hydrodynamic equation, through some sort of renormalization function. By coupling the {\em{constant-rate totally asymmetric zero-range}} with the totally asymmetric simple exclusion, we derive limiting laws for more general initial conditions.
\end{abstract}

\maketitle
\thispagestyle{empty}

\section{Introduction}
The microscopic dynamics of interacting particle systems consists in having particles distributed on a lattice, performing random walks according to some local restriction \cite{Spi.}. These systems have {\em hydrodynamical behavior} if there exists a space and time scaling in which their conserved thermodynamical quantities are described by some
partial differential equations - the {\em hydrodynamic equations} \cite{K.L.}. An interesting problem of those systems is to understand the relation between objects at the particle level and the corresponding macroscopic hydrodynamic equations.
One can often take advantage of the known results on the hydrodynamic equations in order to obtain useful information about the underlying particle system. On the other hand, the study of a particle system has given answers about  the qualitative behavior of the solutions of its hydrodynamic equations \cite{F.K.}.\\

The problem addressed in this paper consists in extending the results of \cite{F.K., F.G.M.} to zero-range processes. In \cite{F.K., F.G.M.} it is studied the asymptotic behavior of second class particles added to a particle system with exclusion dynamics and  the second class particles are related to the characteristics of the corresponding hydrodynamic equation. There, the hydrodynamic equation is the {\em inviscid Burgers equation} and the corresponding particle system is the asymmetric simple exclusion process (asep) that we fully describe below.\\

The processes we consider are one-dimensional asymmetric zero-range processes (azrp) with hydrodynamic equation as given in \eqref{EH-tazrp}, where the flux $F(\cdot)$ is a concave function. The dynamics of the azrp is the following: at each site $x\in\mathbb{Z}$, there is a mean one exponential time clock, after which, a particle at $x$ jumps to $y$ at rate $p(y)g(k_x)/k_x$, where $k_x$ is the number of particles at the site $x$ and $p(\cdot)$ is the single particle transition probability. After that jump, the clocks restart.  Here we assume that $p(1)=p$, $p(-1)=q=1-p\neq 1/2$ and $p(y)=0$ for $|y|>1$. As a consequence, jumps occur only to nearest neighbors. The function $g(\cdot)$ is called the process rate and it  satisfies the conditions described in Section \ref{tazrp}. In the sequel we assume that $p>q$, so that there is a drift to the right. In this process, a jump occurs independently from the number of particles at the destination site, and for that reason it coined the name zero-range. When $p=1$, the process is called totally asymmetric zero-range process (tazrp)  and for $g(k)=\textbf{1}{\{k\geq{1}\}}$ the process is the {\em{constant-rate asymmetric zero-range process}}. In last case, the process rate does not depend on the number of particles $k_x$.\\

Since we are restricted to the one-dimensional setting, we can couple the constant-rate tazrp with the totally asymmetric simple exclusion process (tasep) that was analyzed in \cite{F.K.}. The dynamics of the asep is defined as follows: at each site $x\in{\mathbb{Z}}$ there is a mean one exponential time clock, after which, a particle at $x$ jumps to $y$ at rate $p(y)$ with $p(\cdot)$ as given above. The jump occurs if and only if the destination site is empty, otherwise the particle does not move. When $p=1$, the process is the tasep. By coupling both processes: the constant-rate tazrp and the tasep, we are able to confirm the results that we prove independently for the constant-rate tazrp, from the results previously obtained in \cite{F.K.}. The main difficulty we face is that, contrarily to the tasep in which there is at most one particle per site, on the tazrp there can be an unbounded number of particles at any given site.\\

Now we explain the features of the models that we need in order to achieve our goals. We start by describing the hydrodynamic limit scenario.
Since \cite{A.V.,R.}, it is known that the empirical measures (see Section \ref{sechl}) associated to these processes, converge (in probability) to a deterministic measure whose density is the unique entropy solution of a hyperbolic conservation law with concave flux as given in \eqref{EH-tazrp}. The hydrodynamic limit for these processes was set under two different approaches: for both processes in \cite{R.} using the Entropy method and for a general set of initial measures associated to a profile; and for the asep in \cite{J.Y.} using the Relative Entropy method, under a more restrictive set of initial measures. We notice that it is not difficult to show the hydrodynamic limit for the azrp invoking the same arguments as in \cite{J.Y.}. We do not prove this result here, since it is basically a reproduction (with the proper modifications) of the proof presented in \cite{J.Y.}. As a consequence of that result, we obtain the hydrodynamic limit for the azrp starting from measures of slowly varying parameter associated to flat profiles as well as step decreasing profiles. Besides requiring an hyperbolic hydrodynamic equation with concave flux, the dynamics of the system has to conserve the number of particles. This is true for both exclusion and zero-range dynamics. As we have seen above, in these models, particles simply move along the one-dimensional lattice according to a prescribed rule. Moreover, the particle systems have to be attractive, in order that when coupling two copies of the same process, the number of discrepancies between the initial configurations does not increase as time evolves. Also, this is true for both exclusion and zero-range dynamics, for details see Section \ref{attractiveness}.\\

The main goal of this paper is to analyze the asymptotic behavior of the azrp starting from a configuration with
particles with different degree of class. We mainly consider first or second class particles, but we define the interaction dynamics between particles with all $m\in{\mathbb{N}}$ \footnote{Throughout the article we use the notation $\mathbb{N} = \{1,2,\dots\}$ and $\mathbb{N}_0 = \mathbb N \cup \{0\}$.} degree of class. We start by explaining the dynamics of these particles in the asep and then we describe it in the azrp.
Distribute  particles on $\mathbb{Z}$ (at most one at each site) and at the initial time label each one of them as a $m$-th class particle, with $m\in{\mathbb{N}}$. The dynamics of these particles depend on their label according to the following rule: a $m$-th class particle sees particles
with degree of class less (resp. greater) than $m$ as particles (resp. holes). So, in the asep, a $m$-th class particle moves to the right neighboring site (at rate $p$) if it is empty (or occupied with a particle with degree of class greater than $m$) or if it is occupied with a particle with degree of class less than $m$ that attempts to jump to the left; or moves to the left neighboring site (at rate $q$) if it is empty (or occupied with a particle with degree of class greater than $m$) or if it is occupied with a particle with degree of class less than $m$ that attempts to jump to the right. In the azrp, due to the possibility of having more than one particle per site, the dynamics of particles with different degree of class in this process is completely different from their dynamics in the asep. In the azrp the interaction dynamics is the following.
 Suppose that initially at $x$ there are infinitely many particles. We label each one of them as a $m$-th class particle, with $m\in\mathbb{N}$. If the clock at $x$ rings, the first class particle has priority to jump and moves either to the right or to the left neighboring site. We notice that, the remaining particles at $x$ do not move, nor the particles at $x+1$ or $x-1$. More generally, if at a given site there is no particle with degree of class less than $m$, then after a ring of the clock at this site, the $m$-th class particle is the first jumping particle. As a consequence, the higher the degree of class of a particle, the lower is its priority to jump.\\

Now, we recall some results obtained for the asep in \cite{F.K., F.G.M.}, that we generalize to the azrp in a similar setting.
In \cite{F.K.}, it was derived a Law of Large Numbers (L.L.N.) for the position of a second class particle in tasep starting from $\nu_{\alpha,\beta}$,  the Bernoulli product measure with parameter $\alpha$ (resp. $\beta$) on negative (resp. positive) sites, with $\alpha>\beta$ and with a second class particle at the origin, see Theorem \ref{convergence to uniform of scp of tasep} in Section \ref{sec coupling}. Later, this result was extended to the asep in \cite{F.G.M.}. The speed of the second class particle in the tasep was studied in \cite{A.A.V.} by analyzing the invariant measures of the multi-class tasep. In \cite{M.G.}, this L.L.N. was derived in a stronger sense, namely, almost surely. In \cite{F.P.}, for $\nu_{1,0}$,  last result was also derived by mapping the tasep to a last passage percolation model, since the second class particle can be seen as an interface between two random growing clusters (see \cite{F.G.M.} or \cite{F.P.} for details). We notice that this mapping with a last passage percolation model does not fit the azrp and there is no knowledge on the invariant measures for the multi-class azrp. So, our approach consists in using coupling arguments similar to those of \cite{F.K.,F.G.M.}. Also in \cite{BKS}, were obtained bounds on the variance of a second class particle added to a  particle system by showing that it superdiffuses. The result holds for a large class of particle systems, including some types of azrp. For details we refer to \cite{BKS} and references therein. More recently, in \cite{C.P.} by mapping the tasep to a last passage percolation model and using Busemann functions for the percolation model, the authors derive L.L.N. for the second class particle in the tasep for general initial conditions in the rarefaction fan. \\

  Here, first, we start the constant-rate tazrp from $\mu_{\rho,0}$, with $0<{\rho}\leq{\infty}$ (see  \eqref{murholambdageo}) and we add infinitely many second class particles at the origin. We prove that the weighted sum of joint
probabilities for second class particles sharing the same site is convergent and we compute its limit. In case of partially asymmetric jumps the result is left open. Second, we start the azrp from $\tilde{\xi}_*$ (see Section \ref{sec: statement}) and we denote by $X_2(t)$ the position of the second class particle at time $t$.  We derive a L.L.N. for $X_2(t)$ and also a L.L.N. for the current of first class particles that cross over the second class particle up to time $t$. We believe that last result holds for general particle systems with the required features stated above and for general initial conditions in the rarefaction setting.\\

Finally,  we couple the constant-rate tazrp with the tasep and we extend the previous results to other initial conditions. We also reprove some of the results that we have obtained independently for the constant-rate tazrp. We notice that our coupling between both processes is completely new and relates the constant-rate zero-range with the tasep, in presence of particles with different degree of class. However, we notice that it does not work for partially asymmetric jumps. We remark that using the coupling between the constant-rate tazrp and the tasep,  one can translate the known results for the tasep recently derived for example in \cite{C.P.} to the  constant-rate tazrp. This is subject for future work.
Also, under a coupling argument and invoking Theorem 2.3 a) of \cite{F.G.M.}, starting the constant-rate tazrp from a configuration with a second class particle at the origin, a third class particle at the site $1$, all negative sites occupied by infinitely many first class particles and all positive sites empty, we prove that the probability that the second class particle overtakes the third class particle equals $2/3$. We notice that, a crucial ingredient in the proof of Theorem 2.3 a) in \cite{F.G.M.} is the fact that the number of holes is a conserved quantity for the exclusion dynamics. Since this is not true for the zero-range, that argument of proof does not fit our purposes.
All open problems mentioned in this paper are subject for future work.\\

Here follows an outline of this paper. In the second section, we introduce the azrp and asep; we state their hydrodynamic limit; we present their invariant measures that are translation invariant; we compute the characteristics of hyperbolic conservation laws with concave flux; we describe the couplings and the dynamics of discrepancies that we use along the paper and finally, we state the main results.
In the third section, we study the behavior of second class particles sharing the same site on the constant-rate tazrp starting from $\mu_{\rho,0}$ with $0<{\rho}\leq{\infty}$ and with infinitely many second class particles added to the origin. In the fourth section, starting the azrp from $\tilde{\xi}_*$, we prove the L.L.N. for the position of the second class particle and also for the current of first class particles that cross over the second class particle up to the time $t$. In the fifth section, by coupling the constant-rate tazrp with the tasep and invoking the results known for tasep, we deduce the corresponding results for the tazrp.
\section{Statement of results}
\subsection{The dynamics} \label{seczr}
\subsubsection{Asymmetric zero-range.} \label{tazrp}

 Let $\{\xi_t; t \geq 0\}$ be the one-dimensional azrp, a continuous time Markov process with state space $\mathbb{N}_{0}^{\mathbb{Z}}$. For a configuration $\xi\in{\mathbb{N}_{0}^{\mathbb{Z}}}$ and for a site $x\in{\mathbb{Z}}$, $\xi(x)$ denotes the number of particles at the site $x$. A function $f:\mathbb{N}_{0}^{\mathbb{Z}}\rightarrow{\mathbb R}$ is said to be local if it depends on $\xi$ only through its values on a finite number of sites.
  In this process, after a mean one exponential time, a particle at the site $x$ jumps to $x+1$ or $x-1$ at rate $pg(\xi(x))/\xi(x)$ or $qg(\xi(x))/\xi(x)$, respectively. So that the jump rate is independent from the number of particles at the destination site. In what follows  $g:\mathbb{N}_0 \to
\mathbb{ R_+}$ is a function satisfying: $g(0)=0$, $g(k)>0$ for all $k\in{\mathbb{N}}$,  Lipschitz: $\sup_{k\in{\mathbb{N}_0}} |g(k+1)-g(k)|<\infty$ and non-decreasing: $g(k+1)-g(k)\geq{0}$ for all $k\in{\mathbb{N}_0}$. Now we explain the imposed conditions on $g(\cdot)$. The first condition, ensures that a jump at a given site occurs if there is at least a particle at that site. The second condition guarantees that once having a particle at a site there is a positive probability of the occurrence of jump, so that the dynamics is non-degenerate. The third condition guarantees that the process is well defined. The last condition ensures the attractiveness property of the system, which we describe in Section \ref{attractiveness}. We refer to \cite{A.} for details on the construction of this process.  Its infinitesimal generator is defined on local functions
$f:{\mathbb{N}_0^{\mathbb{Z}}\rightarrow{\mathbb{R}}}$ as
\begin{equation*}\label{generator for zr}
\mathcal{L} f(\xi)=\sum_{x\in{\mathbb{Z}}}pg(\xi(x))[f(\xi^{x,x+1})-f(\xi)]+qg(\xi(x))[f(\xi^{x,x-1})-f(\xi)],
\end{equation*}
where for $z\neq{x,y}$, $\xi^{x,y}(z)=\xi(z)$, $\xi^{x,y}(x)=\xi(x)-1$ and $\xi^{x,y}(y)=\xi(y)+1$.

\subsubsection{Asymmetric simple exclusion.}
 Let $\{\eta_t; t \geq 0\}$ be the one-dimensional asep, a continuous time Markov process with state space $\{0,1\}^{\mathbb{Z}}$. In this process, particles evolve on $\mathbb{Z}$ according
to interacting random walks with an exclusion rule which prevents having more than a particle per site. Its dynamics is defined as follows.  After a mean one exponential time, a particle at the site $x$ jumps to $x+1$ or  $x-1$ at rate $p$ or $q$, respectively. The jump occurs if and only if the destination
site is empty, otherwise the particle does not move and the clocks restart. For a configuration $\eta\in \{0,1\}^{\mathbb{Z}}$ and for $x\in{\mathbb{Z}}$, $\eta(x)$ denotes the quantity of particles at the site $x$.

We notice that, both dynamics introduced above are {\em{particle-conservative}}, since particles are not created nor destroyed, they simply move in the one-dimensional lattice according to a prescribed rule.

\subsubsection{Attractiveness.}\label{attractiveness}
Now, we recall the notion of attractiveness for the zero-range process.
In the state space $\mathbb{N}_0^{\mathbb{Z}}$ there is a partial order between configurations that is defined as follows. For two configurations $\zeta,\tilde{\zeta}\in{\mathbb{N}_0^{\mathbb{Z}}}$, we say that $\zeta\leq{\tilde{\zeta}}$ if $\forall{x\in{\mathbb{Z}}}$ $\zeta(x)\leq{\tilde{\zeta}(x)}$. This partial order induces the corresponding stochastic order on the distributions of the process. Let $\mathcal{F}$ denote the set of monotone functions, namely the set of functions $f:\mathbb{N}_0^{\mathbb{Z}}\rightarrow{\mathbb{R}}$ such that $f(\zeta)\leq{f(\tilde{\zeta})}$, whenever $\zeta\leq{\tilde{\zeta}}$. For two probability measures $\mu$ and $\tilde \mu$ defined on $\mathbb{N}_0^{\mathbb{Z}}$, we say that $\mu\leq{\tilde{\mu}}$ if for all $f\in{\mathcal{F}}$ it holds that $\int fd\mu\leq{\int fd\tilde{\mu}}$.

It was first proved by \cite{A.} that for $g(\cdot)$ non-decreasing, the corresponding zero-range process is attractive. This means that given  $\zeta\leq{\tilde{\zeta}}\in{\mathbb{N}_0^{\mathbb{Z}}}$, it is possible to construct a coupling of the zero-range process $(\zeta_{t},\tilde{\zeta}_{t})$ starting from $(\zeta,\tilde{\zeta})$ such that $\zeta_{t}\leq{\tilde{\zeta}_{t}}$,  $\forall{t>0}$. For details we refer to \cite{A.}, \cite{K.} or \cite{L.}. Since we imposed that condition on $g(\cdot)$, our zero-range processes are attractive.  From \cite{L.} it is known that the asep is attractive. As mentioned above, the attractiveness property will be crucial for our conclusions.

\subsubsection{Invariant measures.} \label{invariant measures}
Now we describe a set of invariant measures for zero-range and exclusion processes. We start by the former.

 The zero-range process has invariant measures, which are product, defined on $\mathbb{N}_0^{\mathbb{Z}}$ and  invariant by translations. These measures can be constructed  as follows. For $x\in{\mathbb{Z}}$, $k\in{\mathbb{N}_0}$ and $\varphi>0$, let $\mu_{\varphi}$ be  the product measure with marginal given by
$$\mu_{\varphi}(\xi:\xi(x)=k)=\frac{1}{Z(\varphi)}\frac{\varphi^k}{g(k)!},$$
where $Z(\varphi)=\sum_{k\geq{0}}\frac{\varphi^k}{g(k)!}$, $g(k)!=\prod_{j=1}^{k}g(j)$ and $g(0)!=1$.
 The measures $\mu_{\varphi}$ can be parameterized by the density of particles as follows. Let $R(\varphi):=E_{\mu_\varphi}[\xi(0)]$ and define $\tilde{g}(\cdot)$ as the inverse of $R(\cdot)$. Define $\mu^g_{\rho}:=\mu_{\tilde{g}(\rho)}$. Then, $E_{\mu_\rho^g}[\xi(0)]=\rho$  and $E_{\mu_\rho^g}[g(\xi(0))]=\tilde{g}(\rho).$ Notice that for the constant-rate azrp, $\mu_\rho^g:=\mu_\rho$ is the Geometric
product measure of parameter $\frac{1}{1+\rho}$, that is for $x\in{\mathbb{Z}}$, $k\in{\mathbb{N}_0}$ and $\rho>0$, $\mu_\rho$ has marginal given by:
\begin{equation} \label{geometric product measure}
\mu_{\rho}(\xi:\xi(x)=k)=\Big(\frac{\rho}{1+\rho}\Big)^k\frac{1}{1+\rho}.
\end{equation}

 Now, for $x\in{\mathbb{Z}}$, $k\in{\mathbb{N}_0}$ and   $0<\lambda< \rho< \infty$, let $\mu^{g}_{\rho,\lambda}$ be the product measure with marginal given by
$\mu_{\rho,\lambda}^{g}(\xi:\xi(x)=k)=\mu_\rho^g(\xi:\xi(x)=k)\textbf{1}{\{x\in(-\infty,0)\}}+\mu_\lambda^g(\xi:\xi(x)=k)\textbf{1}{\{x\in{[0,+\infty)}\}}$.  For the constant-rate azrp, $\mu_{\rho,\lambda}^{g}:=\mu_{\rho,\lambda}$ is given by
\begin{equation}\label{murholambdageo}
 \mu_{\rho,\lambda}(\xi:\xi(x)=k)=\left\{
\begin{array}{rl}
\Big(\frac{\rho}{1+\rho}\Big)^k\frac{1}{1+\rho}, & \mbox{if $x<0$}\\
\Big(\frac{\lambda}{1+\lambda}\Big)^k\frac{1}{1+\lambda}, & \mbox{if $x\geq 0$}
\end{array},
\right.
\end{equation}
for $k\in{\mathbb{N}_0}$ and $x\in{\mathbb{Z}}$.
For $\lambda=0$, if $\xi$ is distributed according to $\mu^{g}_{\rho,0}$, then $\xi(x)=0$ for $x\geq{0}$ and  $\xi(x)$ distributed according to $\mu^g_\rho$ for $x< 0$.

Along the paper we also consider the zero-range process starting from configurations with infinitely many particles at a given site. In that case, in order to have the zero-range process well defined we need to assume that $g(\cdot)$ is bounded. For the starting measure $\mu^{g}_{\infty,\lambda}$, if a configuration $\xi$ is distributed according to $\mu^{g}_{\infty,\lambda}$, then $\xi(x)=\infty$ for $x<{0}$ and $\xi(x)$ is distributed according to $\mu^g_\lambda$ for $x\geq 0$. When $\lambda=0$, $\mu^{g}_{\infty,0}:=\mu_{\infty,0}$ gives weight one to the configuration $\tilde\xi$, such that  $\tilde{\xi}(x)=\infty$ for $x<{0}$ and  $\tilde{\xi}(x)=0$ for $x\geq 0$, see the figure below \eqref{ec:1}.

On the asep, for each density $\alpha\in{[0,1]}$, the
Bernoulli product measure  that we denote by $\nu_{\alpha}$, defined on $\{0,1\}^{\mathbb{Z}}$
 is an invariant measure, translation invariant and also parameterized by the density $\alpha$, namely: $E_{\nu_{\alpha}}[\eta(0)]=\alpha$.
For $x\in{\mathbb{Z}}$, $k\in{\{0,1\}}$ and $\alpha\in{[0,1]}$, its marginal is given by $\nu_{\alpha}(\eta:\eta(x)=k)=\alpha^{k}(1-\alpha)^{1-k}$.
For $x\in{\mathbb{Z}}$, $k\in{\{0,1\}}$ and $\alpha, \beta\in{[0,1]}$, let $\nu_{\alpha,\beta}$ be the product measure with marginals $\nu_{\alpha,\beta}(\eta:\eta(x)=k)=\nu_\alpha(\eta:\eta(x)=k)\textbf{1}{\{x\in{(-\infty,0)}\}}+\nu_{\beta}(\eta:\eta(x)=k)\textbf{1}{\{x\in{[0,+\infty)}\}}$.
For more details we refer the reader to \cite{K.L.} or \cite{L.}.

\subsection{Hydrodynamic limit} \label{sechl}
Now we state the hydrodynamic limit for the processes introduced above. Fix a configuration $\xi$, let $\pi^{n}(\xi,du)$ be the empirical measure given by $\pi^{n}(\xi,du)=\frac{1}{n}\sum_{x\in\mathbb{Z}}\xi(x)\delta_{\frac{x}{n}}(du),$ and let $\pi_{t}^{n}(\xi,du)=\pi^{n}(\xi_{t},du)$. Since the work of \cite{R.}, it is known that starting the azrp from a measure $\mu_{n}$ associated to a profile $\rho_{0}:{\mathbb{R}}\rightarrow{\mathbb{R}}$ and some additional hypotheses (for details see \cite{R.}), if $\pi^{n}_{0}(\xi,du)$ converges to $\rho_{0}(u)du$ in $\mu_{n}$-probability, as $n$ tends to $\infty$, then $\pi^{n}_{tn}(\xi,du)$ converges to $\rho(t,u)du$ in $\mu_{n}S_{tn}$-probability, as $n$ tends to $\infty$, where $S_{t}$ is the semigroup corresponding to $\mathcal{L}$ and $\rho(t,u)$ is the unique entropy solution of the hyperbolic conservation law:
\begin{equation} \label{EH-tazrp}
\partial_{t}\rho(t,u)+\nabla F(\rho(t,u))=0,
\end{equation}
with initial condition $\rho(0,u):=\rho_0(u)$ for all $u\in{\mathbb{R}}$.
 We notice that $F(\rho)$ corresponds to the mean (with respect to the invariant measure $\mu_\rho^g$) of the instantaneous current at the bond $\{0,1\}$. Since jumps occur to neighboring sites, for a site $x\in{\mathbb{Z}}$ the instantaneous current at the bond $\{x,x+1\}$ is defined as the difference between the process rate to the right neighboring site and the process rate to the left neighboring site.

\begin{itemize}
\item In the azrp, the instantaneous current at the bond $\{x,x+1\}$ is given by $pg(\xi(x))-qg(\xi(x+1))$. As a consequence,
\begin{equation}\label{def j}
F(\rho):=(p-q)\tilde g(\rho).
\end{equation}

     \item In the constant-rate azrp, the instantaneous current at the bond $\{x,x+1\}$ is given by $p\textbf{1}{\{\xi(x)\geq{1}\}}-q\textbf{1}{\{\xi(x+1)\geq{1}\}}$. Recalling that for $\rho>0$, $\mu_\rho$ is given by
         \eqref{geometric product measure} then,  $F(\rho)=(p-q)\frac{\rho}{1+\rho}$.
\end{itemize}

 Also in \cite{R.}, the hydrodynamic limit was derived for the asep, but in this case  $F(\rho)=(p-q)\rho(1-\rho)$, getting in \eqref{EH-tazrp} to the {\em{inviscid Burgers equation}}.

\subsection{Characteristics} \label{secch}
Here we follow \cite{K.K.}.
We describe the characteristics for hyperbolic conservation laws as given in \eqref{EH-tazrp}. Suppose that the flux $F(\cdot)$ is {\em{concave}} and differentiable. A characteristic is  a
trajectory of a point with constant density and if $v_{\rho_{0}}(t,u)$ denotes the position at time $t$ of a point with density $\rho_{0}=\rho(0,u)$, then $\rho(t,v_{\rho_{0}}(t,u))=\rho_{0}$.
Taking the time derivative of last expression it follows that
$\partial_{s}v_{\rho_{0}}(s,u)=F'(\rho_{0})$. Integrating the last expression from time $0$ to time $t$, and noticing that $v_{\rho_{0}}(0,u)=u$, we get to
$v_{\rho_{0}}(t,u)=u+F'(\rho_{0})t.$
So, the characteristics of \eqref{EH-tazrp} are straight lines with slope $F'(\rho_{0})$.

If the initial condition of \eqref{EH-tazrp} is a decreasing step function as
$\rho_{0}^{\rho,\lambda}(u)=\rho\textbf{1}{\{u<{0}\}}+\lambda\textbf{1}{\{u>0\}},$
with $\rho>\lambda$, then the solution of (\ref{EH-tazrp}) is given by
\begin{equation}\label{sol g rate}
 \rho(t,u)=\left\{
\begin{array}{rl}
\rho, & \mbox{if $u<F'(\rho)t$, }\\
\lambda, & \mbox{if $u>F'(\lambda) t$,}\\
(F')^{-1}\Big(\frac{u}{t}\Big), & \mbox{if $F'(\rho)t <u<F'(\lambda) t$}.
\end{array}
\right.
\end{equation}
Now we compute explicitly these solutions for the hydrodynamic equations under consideration.

\begin{itemize}
\item In the azrp, if $\tilde{g}(\cdot)$ is concave, then the solution  is given as in \eqref{sol g rate}, with $F(\cdot)$ defined through \eqref{def j}.

     \item In the constant-rate azrp, since $F(\rho)=(p-q)\frac{\rho}{1+\rho}$ is concave, then the solution is given by
     \eqref{sol g rate} with $(F')^{-1}\Big(\frac{u}{t}\Big)=\frac{\sqrt{(p-q)t}-\sqrt{u}}{\sqrt{u}}.$
\end{itemize}

\subsection{Couplings and discrepancies} \label{sec coupling and discrepancy}
A coupling between two processes consists in a joint realization of those processes. In this paper we use couplings between the azrp starting from different initial configurations in order to study the asymptotic behavior of second class particles.
In Section \ref{sec coupling}  we  couple the constant-rate tazrp with the tasep, in order to obtain some results for the constant-rate tazrp from the results obtained in \cite{F.K., F.G.M.}.

When coupling two copies of the azrp starting from different initial configurations, we make use of the ``basic coupling". This consists of attaching a Poisson clock of parameter one to each site of $\mathbb{Z}$, so that when the clock rings at a site either there is a particle at that site and it moves according to the dynamics described above, or nothing happens. Under this coupling both copies of the process use the same realization of the clocks.

We also make use of the ``particle to particle" coupling. In this case, initially we label the particles in both configurations. Therefore, the $i$-th particle in the first configuration is attached to the $i$-th particle in second configuration. Nevertheless, we can use the same realizations of the clocks attached to sites, but only one of the configurations looks at the clocks. Then, if a clock rings for the $i$-th particle in the first configuration, then the $i$-th particles of both configurations jump.

Now we introduce the notion of discrepancy between two copies of an attractive zero-range process. For that purpose, let $\xi^0$ and $\xi^1$ be two configurations such that $\xi^0(x)=\xi^1(x)$ for all $x\neq{0}$ and at the site zero take $\xi^0(0)=\xi^1(0)+1$. So, at time zero there is only one discrepancy between $\xi^0$ and $\xi^{1}$ that is located at the origin. Recall the partial order defined in Section \ref{attractiveness} and notice that $\xi^0\geq{\xi^1}$. Considering two copies $\xi_t^0$ and $\xi_t^1$ of the zero-range process starting from $\xi^0$ and $\xi^{1}$, respectively, and using the basic coupling, by the {\em{attractiveness property}}, we have that $\xi^0_t\geq{\xi_t^1}$ for all $t\geq{0}$. Moreover, by the {\em{conservation of the number of
particles}}, at each time $t$, there is at most a unique discrepancy between $\xi^0_{t}$ and $\xi^1_{t}$. This means that either there exists a site that we denote by $X_2(t)$, such that $\xi^0_t(X_2(t))=\xi^1_t(X_2(t))+1$ and for all $x\neq{X_2(t)}$, $\xi^0_t(x)=\xi^1_t(x)$; or $\xi^0_t(x)=\xi^1_t(x)$ for all $x\in{\mathbb{Z}}$. This last situation can happen when we consider the zero-range process starting from a configuration with infinitely many particles at a given site and at time $t$ the discrepancy jumps to that site, so that it disappears and the configurations become equal. The attractiveness property of the zero-range processes is crucial for our purposes, since when coupling these processes, the number of discrepancies {\em does not increase}.

As mentioned in the introduction, due to the rigid structure of the state space of the exclusion process, the dynamics of second class particles is very different in zero-range and exclusion type dynamics. On the azrp, when first and second class particles share the same site, the first class particles have priority to jump, so that a second class particle leaves a site $x$ if and only if there is no first
class particles at $x$. On the other hand, for the exclusion type dynamics, the second class particle leaves the site $x$ if one of the neighboring sites is empty, or if a first class
particle jumps to the site occupied by the second class particle, in that case the particles exchange their positions.

Obviously, different initial conditions for a coupled pair of azrp are possible. For example, if in the coupling described above
we take $\xi^0(0)=\xi^1(0)+k$, with $k\in{\mathbb{N}}\setminus{\{1\}}$, then there are $k\neq{1}$ second class particles at the origin. By the attractiveness property together with the conservation of the number of particles, at any time $t$ there is at most $k$ discrepancies between the two copies of the azrp.

\subsection{Statement of results}\label{sec: statement}

\begin{theorem} \label{proposition}
Consider the constant-rate tazrp starting from $\mu_{\rho,0}$ with $0<{\rho}\leq{\infty}$ and at the initial
time add infinitely many second class particles at the origin. At the initial time label the second class particles, from the bottom to the top
and for $j\in{\mathbb{N}}$, let $X_2^{j}(t)$ be the position at time $t$ of the
$j$-th second class particle initially at the origin. Then,

\begin{itemize}
\item for $\rho<\infty$ and $u\in{[1/(1+\rho)^2,1]}$,
\begin{equation*}
\lim_{t\rightarrow{+\infty}}\,\,\sum_{j\geq{1}}\,\,P(X_2^1(t)\geq{ut},\cdots,X_2^{j}(t)\geq{ut})\Big(\frac{\rho}{1+\rho}\Big)^j=\rho(1,u),
\end{equation*}

\item for $\rho=\infty$ and $u\in{[0,1]}$,
\begin{equation*}
\lim_{t\rightarrow{+\infty}}\,\,\sum_{j\geq{1}}\,\,P(X_2^1(t)\geq{ut},\cdots,X_2^{j}(t)\geq{ut})=\rho(1,u),
\end{equation*}
\end{itemize}
where $\rho(t,u)$ is given in \eqref{sol g rate} with $(F')^{-1}\Big(\frac{u}{t}\Big)=\frac{\sqrt{(p-q)t}-\sqrt{u}}{\sqrt{u}}.$
\end{theorem}

In the same spirit as in \cite{F.K., F.G.M.}, we establish the L.L.N. for a single second class particle initially at the origin for the azrp starting from $\mu_{\infty,0}$. The initial configuration is denoted by $\tilde{\xi}_*$ and corresponds to

\hspace{5.09cm}\vdots

\hspace{4.9cm}
$\CIRCLE$

\hspace{4.9cm}
$\CIRCLE$

\hspace{4.5cm}
$\ldots\CIRCLE \underline{\circledast} \Circle  \Circle \Circle \Circle\ldots$ \hspace{0.8cm} $\tilde{\xi}_*$

 \vspace{1mm}

Above first class
particles are represented by $\CIRCLE$, holes by $\Circle$ and the second class particle by ${\circledast}$. In due course, whenever we underlined a particle, we mean that it is positioned at the origin. As a consequence of the previous result and under the same initial configuration,  we also derive a L.L.N. for the current of first class particles that cross over the second class particle up to the time $t$, that we denote by  $J_{2}^{zr}(t)$. Last current  corresponds to the number of first class particles that jump from $X_2(s)-1$ to $X_2(s)$ minus the number of first class particles that jump from $X_2(s)$ to $X_2(s)-1$ for $s\in{[0,t]}$. Due to our initial condition, this corresponds to the number of first class particles that are at the right hand side (and at the same site) of $X_{2}(t)$ at time $t$:
\begin{equation} \label{current 2cp}
J_{2}^{zr}(t)=\sum_{x\geq{X_{2}(t)}}\xi_{t}(x).
\end{equation}
\begin{theorem} \label{LLN with g}
Consider an azrp such that $\tilde{g}(\cdot)$ is concave and
\begin{equation}\label{limit condition}
\exists \,\ell>0: \hspace{0,2cm} \lim_{\rho\rightarrow{\infty}}\tilde{g}(\rho)=\ell<+\infty.
\end{equation}
Start the azrp from $\tilde{\xi}_*$  and denote by $X_{2}(t)$ the position of the second class particle at time $t$. Then
\begin{equation*}
\lim_{t\rightarrow{+\infty}}\frac{X_{2}(t)}{t}=X, \quad\ \textrm{in  distribution}
\end{equation*}
where $X$ is distributed according to $$F_{X}(u)=1-\frac{F(\rho(1,u))}{p \ell},$$ for $u\in{[0,p-q]}$, where $F(\cdot)$ was defined in \eqref{def j} and $\rho(t,u)$ is given in \eqref{sol g rate}. Moreover,

\begin{itemize}
\item for $p=1$,
$$\lim_{t\rightarrow{+\infty}}\frac{J_{2}^{zr}(t)}{t}=F((F')^{-1}(X))-(F')^{-1}(X)X,\quad\ \textrm{in  distribution}$$
\item for $p\neq{1}$, $$\lim_{t\rightarrow{+\infty}}\frac{J_{2}^{zr}(t)}{t}=+\infty.$$
\end{itemize}
\end{theorem}

\begin{remark}
We notice that if one takes $p\neq{1}$ in the previous theorem, then the limiting distribution has an atom at $u=0$. More precisely, with probability $q/p$, the second class particle stays at the origin. This is related to the fact that starting the azrp from $\tilde{\xi}_*$, if the second class particle jumps to the site $-1$ which has infinitely many first class particles, then the second class particle gets trapped at that site and for that reason the current diverges.
\end{remark}

For the constant-rate azrp,  $\tilde{g}(\rho)=\frac{\rho}{1+\rho}$, therefore $\tilde g(\cdot)$ is concave and $\ell=1$. Since in this case $\rho(t,u)$ is given as in \eqref{sol g rate} with $(F')^{-1}\Big(\frac{u}{t}\Big)=\frac{\sqrt{(p-q)t}-\sqrt{u}}{\sqrt{u}}$, we conclude from the previous result that:

\begin{corollary} \label{LLN non couling}
Consider the constant-rate azrp starting from $\tilde{\xi}_*$  and denote by $X_{2}(t)$ the position of the second class particle at time $t$. Then
\begin{equation*}
\lim_{t\rightarrow{+\infty}}\frac{X_{2}(t)}{t}=X, \quad\ in \hspace{0,1cm} distribution
\end{equation*}
where $X$ is distributed according to $$F_{X}(u)=\frac{q+\sqrt{(p-q)u}}{p},$$ for $u\in{[0,p-q]}$. Moreover,
\begin{itemize}
\item for $p=1$, $$\lim_{t\rightarrow{+\infty}}\frac{J_{2}^{zr}(t)}{t}=(1-\sqrt{X})^2, \quad\ \textrm{in  distribution}$$

\item for $p\neq{1}$,
$$\lim_{t\rightarrow{+\infty}}\frac{J_{2}^{zr}(t)}{t}=+\infty.$$
\end{itemize}

\end{corollary}

By coupling the constant-rate tazrp with the tasep we are able to generalize the statement of the previous corollary. This is the content of next theorem:

\begin{theorem} \label{convergence on the azrp}
Consider the constant-rate tazrp starting from $\mu_{\rho,\lambda}$, with $0\leq{\lambda}<{\rho}\leq{\infty}$ and at the initial time add a second class particle at the origin. Let
 $X_{2}(t)$ denote the position at time $t$ of the second class particle. Then
\begin{equation*} \label{2cptazrpconv}
\begin{split}
&\lim_{t\rightarrow{+\infty}}\frac{X_{2}(t)}{t}=\Big(\frac{1+\mathcal{U}}{2}\Big)^2, \quad\ \textrm{almost surely},\\
&\lim_{t\rightarrow{+\infty}}\frac{J_{2}^{zr}(t)}{t}=\Big(\frac{1-\mathcal{U}}{2}\Big)^2, \quad\ \textrm{almost surely},
\end{split}
\end{equation*}
where $\mathcal{U}$ is uniformly distributed on $\Big[\cfrac{1-\rho}{1+\rho}\,,\,\cfrac{1-\lambda}{1+\lambda}\,\Big]$.
\end{theorem}

\section{Proof of Theorem \ref{proposition}}
\begin{proof}

We start by considering the case $\rho<\infty$. Fix a configuration $\xi\in{\mathbb{N}_0^{\mathbb{Z}}}$ and denote by $J^{u}_{t}(\xi)$ the current of particles that cross the time dependent line $ut$ during the time interval $[0,t]$. Then $J^{u}_{t}(\xi)$ is the number of particles of $\xi$ that are at left of the
origin (including it) at time $0$ and are at the right of $ut$ at time $t$, minus the number of particles of $\xi$ that are strictly at the
right of the origin at time $0$ and are at left of $ut$ (including it) at time $t$.  Fix a configuration $\xi$. For each site $x\in{\mathbb{Z}}$, we label the $\xi(x)$ particles from the bottom to the top, in such a way that the first particle is the one being at the bottom and the $\xi(x)$-th particle being the
one at the top. For  $x\in{\mathbb{Z}}$ and $z=1,\cdots,\xi(x)$, we denote by $X^{x,z}_{t}(\xi)$ the position at time $t$ of the z-th tagged particle initially at the site $x$.
We notice that we use the same definition of the current as given for the exclusion process in \cite{F.F.1}, but extended to the zero-range process. Then:
\begin{equation} \label{def of current through ut}
J^{u}_{t}(\xi)=\sum_{x\leq{0}}\sum_{z=1}^{\xi(x)}\textbf{1}{\{X^{x,z}_{t}(\xi)>ut\}}-\sum_{x>0}\sum_{z=1}^{\xi(x)}\textbf{1}{\{X^{x,z}_{t}(\xi)\leq{ut}\}}.
\end{equation}

For $x\in{\mathbb{Z}}$, denote by $\tau_{x}$ the space translation by $x$, so that for $y\in{\mathbb{Z}}$, $\tau_x\xi(y):=\xi(y-x)$. This definition extends naturally to local functions $f$ as $\tau_{x}f(\xi)=f(\tau_x\xi)$ and for probability measures $\mu\in{\mathbb{N}_0^{\mathbb{Z}}}$ as $\int f(\xi)\tau_{x}\mu(d\xi)=\int f(\tau_x\xi)\mu(d\xi)$.

 Now, we compute in two different ways:
$$\int E\Big[J^{u}_{t}(\xi)\Big]\tau_1\mu_{\rho,0}(d \xi)-\int E\Big[J^{u}_{t}(\xi)\Big]\mu_{\rho,0}(d \xi).$$ Fix $\xi^0$ and $\xi^1$ distributed according to $\tau_1\mu_{\rho,0}$ and $\mu_{\rho,0}$, respectively, such that $\xi^0(x)=\xi^1(x)$ for all $x\neq{0}$ and at the origin  $\xi^0(0)$ is Geometric distributed with parameter $1/(1+\rho)$ and $\xi_0^1(0)=0$. Let $\xi^{0}_t$ and $\xi^{1}_t$ be two copies of the constant-rate tazrp starting from  $\xi^0$ and $\xi^1$, respectively. For  a coupling $\bar{P}$ of $\xi^{0}_t$ and $\xi^{1}_t$ starting from  $\xi^0$ and $\xi^1$, respectively,
last expression equals to
\begin{equation} \label{dif exp}
\int d\bar{\mu}(\xi^0,\xi^1)\bar{E}\Big[J^{u}_{t}(\xi^0)-J^{u}_{t}(\xi^1)\Big],
\end{equation}
where $\bar{E}$ is the expectation with respect to $\bar{P}$ and $\bar{\mu}$ is the coupling measure with marginals: $\bar\mu((\xi^0,\xi^1):\xi^0)=\tau_1\mu_{\rho,0}(\xi^0)$ and
 $\bar\mu((\xi^0,\xi^1):\xi^1)=\mu_{\rho,0}(\xi^1).$ Now, the difference between the fluxes $J^{u}_{t}(\xi^0)-J^{u}_{t}(\xi^1)$ is non zero  if and only if $\xi^0$ and $\xi^1$ have at least one discrepancy at the origin. Therefore,
we split last event on the number of discrepancies between the configurations and we condition the previous integral to $D_{k}:=\{(\xi^0,\xi^1): \xi^0(0)=\xi^1(0)+k\}$. Then, last expression can be written as
\begin{equation} \label{dif exp new}
\sum_{k\geq{1}}\int d\bar{\mu}(\xi^0,\xi^1)\bar{E}\Big[\Big(J^{u}_{t}(\xi^0)-J^{u}_{t}(\xi^1)\Big)\Big|D_{k}\Big]\bar{\mu}(D_{k}).
\end{equation}
Notice that $\bar{\mu}(D_{k})=\rho^k/(1+\rho)^{k+1}$. At this point, the previous integral is restricted to configurations $\xi^0$ and $\xi^1$ with exactly $k$ discrepancies at the origin.
 Now, we consider the basic coupling described in Section \ref{sec coupling and discrepancy} in order to compute \eqref{dif exp new}. Under that coupling, particles in the configurations $\xi^0$ and $\xi^1$ use the same realizations of the clocks in order to jump. So, in $D_k$ the difference between the currents $J^{u}_{t}(\xi^0)$ and $J^{u}_{t}(\xi^1)$ can vary from $1$ to $k$ and it depends on the relative position of the $k$ discrepancies with respect to the time dependent line $ut$. Therefore,  we can write  \eqref{dif exp new} as
\begin{equation}\label{dif exp new1}
\sum_{k\geq{1}}\Big(\sum_{j=1}^{k}jP_{\bar{\mu_{k}}}(J^{u}_{t}(\xi^0)-J^{u}_{t}(\xi^1)=j)\Big)\frac{\rho^k}{(1+\rho)^{k+1}},
\end{equation}
 where $\bar{\mu}_{k}$ is the coupling measure conditioned to $D_{k}$: $\bar\mu_k(\cdot)=\bar\mu(\cdot|D_k)$.

 It is easy to check that the discrepancies behave as second class particles, see Section \ref{sec coupling and discrepancy}, so in order to compute the previous expression we need to recall the dynamics of second class particles. When first and second class particles share a same site and the clock at that site rings, the first class particles have priority to jump. When at a site there are only second class particles, they are labeled from the bottom to the top and the first jumping second class particle is the one with the lowest label. In $D_k$, at the initial time in the configuration $\xi^0$, we label the $k$ second class particles from the bottom to the top, and for $j\in{\{1,\cdots,k\}}$ we denote by $X_{2}^j(0)$ the $j$-th second class particle. If the clock at the origin rings, then the first jumping second class particle is $X_{2}^1(0)$. If the next clock to ring is the one at the origin, then $X_{2}^2(0)$ is the first jumping second class particle and the rest of them waits a new random time. As a consequence, the $j$-th second class particle sees the second class particles with label greater (resp. smaller) than $j$ as holes (resp. particles). For $j\in{\{1,\cdots,k\}}$, let $X_2^{j}(t)$ denote the position at time $t$ of the $j$-th second class particle initially at the origin.

Suppose that in $D_k$, the difference between the currents equals to $j\in{\{1,\cdots,k\}}$. Then, this happens if and only if the first $j$-th second class particles initially at the origin are at the right hand side of $ut$ at time $t$ and the rest of them are at the left hand side of $ut$, that is $\{X_2^{j}(t)\geq{ut}>X_2^{j+1}(t)\}$. This is the key point in the proof where we
 use the {\em{total asymmetry}} of jumps. In presence of partial asymmetric jumps the second class particles do not preserve their order. Then \eqref{dif exp new1} can be written as
\begin{equation} \label{eq sum2}
\sum_{k\geq{1}}\Big(\sum_{j=1}^{k-1}j\,p_{k}(j,t)+\,k\,P_{\bar{\mu}_{k}}(X_2^{k}(t)\geq{ut})\Big)\frac{\rho^k}{(1+\rho)^{k+1}},
\end{equation}
where $p_{k}(j,t)=P_{\bar{\mu_{k}}}(X_2^{j}(t)\geq{ut}>X_2^{j+1}(t))$. We remark that, as consequence of the dynamics of second class particles, whenever we write the event $\{X_2^j(t)\geq{ut}\}$ it should be understood  as $\{X_2^1(t)\geq{ut},\cdots,X_2^j(t)\geq{ut}\}$.
Now, we notice that
\begin{equation*}
\sum_{j=1}^{k-1}jp_{k}(j,t)=\sum_{j=1}^{k-1}P_{\bar{\mu}_{k}}(X_2^{j}(t)\geq{ut})-(k-1)P_{\bar{\mu}_{k}}(X_2^{k}(t)\geq{ut}),
\end{equation*}
since by a simple computation $p_{k}(j,t)$ can be written as $P_{\bar{\mu_{k}}}(X_2^j(t)\geq{ut})-P_{\bar{\mu_{k}}}(X_2^{j+1}(t)\geq{ut})$.
Again, the {\em{total
asymmetry}} of jumps is invoked. We point out that, if jumps are partially asymmetric and the difference of the currents is $j$, then in $D_{k}$, one must have
 $j$ of the $k$ discrepancies at the right hand side of $ut$. This can happen a number of $C^{k}_{j}$ events, whose probability we cannot control and for that reason we cannot obtain a similar result to the one above. Collecting these facts together, (\ref{eq sum2}) can be written as
 \begin{equation}\label{dif exp res}
 \sum_{k\geq{1}}\frac{\rho^k}{(1+\rho)^{k+1}}\Big(\sum_{j=1}^{k}P_{\bar{\mu}_{k}}(X_2^{j}(t)\geq{ut})\Big).
 \end{equation}
Now we notice that in the constant-rate tazrp, a particle at a site jumps to the right with a rate independently from the number of particles at its site. As a consequence, $P_{\bar{\mu}_{k}}(X_2^{j}(t)\geq{ut})$ does not depend on $k$, since it corresponds to the probability that the $j$-th second class particle
initially at the origin (having initially the origin with $k$ second class particles), being at the right hand side of $ut$ at time $t$. This event does not depend on the number of $k$ particles but on
$j$ and $t$. When one considers more general zero-range processes defined through $g(\cdot)$, last property is not necessarily true and for that reason one cannot use this argument in that case, nevertheless see in Remark \ref{general g} what we can get for those processes.
From last observations we are able to apply Fubini's theorem to \eqref{dif exp res}. Moreover, noticing that $$\sum_{k\geq{j}}\frac{\rho^k}{(1+\rho)^{k+1}}=\Big(\frac{\rho}{1+\rho}\Big)^j,$$
 we rewrite \eqref{dif exp res} as
\begin{equation*}
\sum_{j\geq{1}}P_{\bar{\mu}}(X_2^{j}(t)\geq{ut})\Big(\frac{\rho}{1+\rho}\Big)^j.
\end{equation*}

Now, we compute (\ref{dif exp}) by coupling $\tau_1\mu_{\rho,0}$ and $\mu_{\rho,0}$ in such a
way that $\xi^0=\tau_{1}\xi^1$. We use the particle to particle coupling,  therefore, we  label the particles in both configurations, only one configuration looks at the clocks and particles are attached by the labels. Then, by \eqref{def of current through ut} we have that
\begin{equation*}
\begin{split}
J^{u}_{t}(\xi^0)-J^{u}_{t}(\xi^1)=&\sum_{x\leq{0}}\sum_{z=1}^{\xi^{0}(x)}\textbf{1}{\{X^{x,z}_{t}(\xi^0)>ut\}}-
\sum_{x>{0}}\sum_{z=1}^{\xi^{0}(x)}\textbf{1}{\{X^{x,z}_{t}(\xi^0)\leq{ut}\}}\\-&
\sum_{x\leq{0}}\sum_{z=1}^{\xi^{1}(x)}\textbf{1}{\{{X}^{x,z}_{t}(\xi^1)>ut\}}+\sum_{x>{0}}\sum_{z=1}^{\xi^{1}(x)}\textbf{1}{\{{X}^{x,z}_{t}(\xi^1)\leq{ut}\}}.
\end{split}
\end{equation*}
 Since $\xi^0=\tau_{1}\xi^1$ and by the particle to particle coupling, we rewrite the right hand side of the previous expression in terms of $\xi^0$ as

 \begin{equation*}
\begin{split}
&\sum_{x\leq{0}}\sum_{z=1}^{\xi^{0}(x)}\textbf{1}{\{X^{x,z}_{t}(\xi^0)>ut\}}-
\sum_{x>{0}}\sum_{z=1}^{\xi^{0}(x)}\textbf{1}{\{X^{x,z}_{t}(\xi^0)\leq{ut}\}}\\-&
\sum_{x\leq{0}}\sum_{z=1}^{\xi^{0}(x+1)}\textbf{1}{\{{X}^{x,z}_{t}(\xi^0)>ut+1\}}+\sum_{x>{0}}\sum_{z=1}^{\xi^{0}(x+1)}
\textbf{1}{\{{X}^{x,z}_{t}(\xi^1)\leq{ut+1}\}}.
\end{split}
\end{equation*}

 We notice that the particle to particle coupling is used here in order to guarantee that the event  $\{{X}^{x,z}_{t}(\xi^1)>ut\}$ (resp. $\{X^{x,z}_{t}(\xi^0)\leq{ut}\}$) corresponds to $\{{X}^{x,z}_{t}(\xi^0)>ut+1\}$ (resp. $\{{X}^{x,z}_{t}(\xi^1)\leq{ut+1}\}$), which is a consequence of particles being attached by the labels. Now, a simple computation shows that
$$J^{u}_{t}(\xi^0)-J^{u}_{t}(\xi^1)=\sum_{x\in{\mathbb{Z}}}\sum_{z=1}^{\xi^{0}(x)}\textbf{1}{\{X^{x,z}_{t}(\xi^0)=ut+1\}}-\xi^{0}(1).$$
 Applying expectation (with respect to $\mu_{\rho,0}$) to the right hand side of last expression, we obtain that (\ref{dif exp}) is equal to $$E_{\mu_{\rho,0}}[\xi^0_{t}(ut+1)]-E_{\mu_{\rho,0}}[\xi^{0}(1)].$$
Now, since $E_{\mu_{\rho,0}}[\xi^{0}(1)]=0$ and by the {\em{convergence to local equilibrium}} (see \cite{A.V.,B.F.}) it follows that:
$$\lim_{t\rightarrow{+\infty}}E_{\mu_{\rho,0}}[\xi^0_{t}(ut+1)]=\rho(1,u),$$
where $\rho(t,u)$ is given in \eqref{sol g rate} with $(F')^{-1}\Big(\frac{u}{t}\Big)=\frac{\sqrt{(p-q)t}-\sqrt{u}}{\sqrt{u}}$. Putting together the previous computations the proof of the first claim ends.

Now we consider the case $\rho=\infty$. By applying the same argument as above, the main difference come from the fact that initially, $\xi^0$ and $\xi^1$ have infinitely many discrepancies at the origin, from where we obtain that \eqref{dif exp} equals to
\begin{equation*}
\sum_{j\geq{1}}j\,P_{\bar{\mu}}(X_2^j(t)\geq{ut}>X_2^{j+1}(t)).
\end{equation*}
Denoting the previous probability by $p(j,t)$ and repeating the same computations as above, we get that \eqref{dif exp} equals to
$$\sum_{j\geq{1}}P_{\bar{\mu}}(X_2^j(t)\geq{ut}).$$ To conclude the proof one just has to follow the same steps as above.
\end{proof}

\begin{remark}\label{general g}
We notice that the previous result also holds for the tazrp starting from $\mu^g_{\rho,\lambda}$ for $g(\cdot)$ satisfying the conditions of Section \ref{seczr} such that $\tilde{g}(\cdot)$ is a concave function. For these processes, we cannot use Fubini´s Theorem as we did below \eqref{dif exp res}, because jumps can depend on the number of particles at a site. Therefore, under the same conditions of Theorem \ref{proposition} we have that
\begin{itemize}
\item for $\rho<\infty$ and $u\in{\Big[\tilde{g}'(\rho),1\Big]}$,
\begin{equation*}
\begin{split}
\lim_{t\rightarrow{+\infty}}\,\,\sum_{k\geq{1}}\Big(\sum_{j=1}^k\,\,P(X_2^1(t)\geq{ut},\cdots,X_2^{j}(t)\geq{ut})\Big)\frac{(\tilde{g}(\rho))^{k}}{Z(\tilde{g}(\rho))g(k)!}=\rho(1,u),
\end{split}
\end{equation*}

\item  if $g(\cdot)$ satisfies \eqref{limit condition}, for $\rho=\infty$ and $u\in{[0,1]}$,
\begin{equation*}
\lim_{t\rightarrow{+\infty}}\,\,\sum_{j\geq{1}}\,\,P(X_2^1(t)\geq{ut},\cdots,X_2^{j}(t)\geq{ut})=\rho(1,u),
\end{equation*}
\end{itemize}
where $\rho(t,u)$ is given in \eqref{sol g rate}.
\end{remark}

\begin{remark} \label{asymptotic independence}
 Here we notice that if $\rho<\infty$, then  for $u\in{[1/(1+\rho)^2,1]}$ it holds $$\lim_{t\rightarrow{+\infty}}\sum_{j\geq{1}}P(X_2^1(t)\geq{ut},\cdots,X_2^{j}(t)\geq{ut})\Big(\frac{\rho}{1+\rho}\Big)^j=
 \lim_{t\rightarrow{+\infty}}\sum_{j\geq{1}}\Big(P(X_2(t)\geq{ut})\Big)^j\Big(\frac{\rho}{1+\rho}\Big)^j,$$
which is equal to $\rho(t,u)$ given in \eqref{sol g rate} with $(F')^{-1}\Big(\frac{u}{t}\Big)=\frac{\sqrt{(p-q)t}-\sqrt{u}}{\sqrt{u}}.$

 For that purpose notice that by Theorem \ref{convergence on the azrp} we have that
\begin{equation}\label{r1}
\lim_{t\rightarrow{+\infty}}P(X_2(t)\geq{ut})=P\Big(\Big(\frac{1+\mathcal{U}}{2}\Big)^2\geq{u}\Big)=\Big(\frac{1+\rho}{\rho}\Big)(1-\sqrt u),
\end{equation}
for $2\sqrt u-1\in[(1-\rho)(1+\rho),1]$ which is equivalent to $u\in{[1/(1+\rho)^2,1]}$. Then, for  $u\in[1/(1+\rho)^2,1]$:
\begin{equation*}
\begin{split}
\lim_{t\rightarrow{+\infty}}\sum_{j\geq{1}}\Big(P(X_2(t)\geq{ut})\Big)^j\Big(\frac{\rho}{1+\rho}\Big)^j=&\sum_{j\geq{1}}
\lim_{t\rightarrow{+\infty}}\Big(P(X_2(t)\geq{ut})\Big)^j\Big(\frac{\rho}{1+\rho}\Big)^j\\
=&\sum_{j\geq{1}}
(1-\sqrt{u})^j\\=&\frac{1-\sqrt u}{\sqrt u}\\=&\rho(1,u).
\end{split}
\end{equation*}
In the first equality we interchanged the summation with the limit, which is possible since it is a convergent Geometric series; and in the second equality we used \eqref{r1}. Comparing this with the statement of Theorem \ref{proposition}, the observation follows.

Analogously, if $\rho=\infty$, then for $u\in{[0,1]}$ it holds $$\lim_{t\rightarrow{+\infty}}\,\,\sum_{j\geq{1}}\,\,P(X_2^1(t)\geq{ut},\cdots,X_2^{j}(t)\geq{ut})=
 \lim_{t\rightarrow{+\infty}}\,\,\sum_{j\geq{1}}\Big(\,\,P(X_2(t)\geq{ut})\Big)^j=\rho(t,u).$$
  For that purpose it is enough to notice that \eqref{r1} holds with $\rho/(1+\rho)$ replaced by $1$ for $u\in{[0,1]}$, from where we get the second equality in the previous expression, the rest follows by Theorem \ref{proposition}.
\end{remark}

\section{Proof of Theorem \ref{LLN with g}}

\begin{proof}
 For $\xi\in{\mathbb{N}_0^{\mathbb{Z}}}$, let $J_{t}^{u}(\xi)$ be the current of first particles that cross the time dependent line $ut$ up to time $t$. If we suppose that $ut$ is an integer, then $J_{t}^{u}(\xi)$ is the number of particles that jump from $ut-1$ to $ut$ minus the number of particles that jump from $ut$ to $ut-1$.
 Consider the azrp starting from $\tilde{\xi}$ introduced in Section \ref{invariant measures}. Then, we have that $J_{t}^{u}(\tilde{\xi})=\sum_{x\geq{ut}}\tilde{\xi}_{t}(x).$
 Without lost of generality we suppose that $ut$ is an integer number, otherwise a similar computation can be done, see the proof of Theorem 2.2 of \cite{F.G.M.} for details.

By the Kolmogorov backwards equation we have that
\begin{equation}\label{ec:1}
\frac{d}{dt}E_{\mu_{\infty,0}}[J_{t}^{u}(\tilde{\xi})]=E_{\mu_{\infty,0}}[\mathcal{L}J_{t}^{u}(\tilde{\xi})]=E_{\mu_{\infty,0}}[p\,g(\tilde{\xi}_t(-1))(J_{t}^{u}(\tilde{\xi}^{-1,0})-J_{t}^{u}(\tilde{\xi}))].
\end{equation}

Now we compute $\frac{d}{dt}E_{\mu_{\infty,0}}[J_{t}^{u}(\tilde{\xi})]$ in two different ways.
At first, notice that since for $\tilde \xi$ there is no particle at the origin, when a particle jumps from $-1$ to $0$, then there is only one
discrepancy at the origin between the two configurations  $\tilde{\xi}$ and $\tilde{\xi}^{-1,0}$, see the figure below:

\vspace{2mm}

\hspace{2.65cm}\vdots   \hspace{5.3cm}\vdots

\hspace{2.47cm}
$\CIRCLE$ \hspace{4.95cm} $\CIRCLE$

\hspace{2.47cm}
$\CIRCLE$ \hspace{4.95cm} $\CIRCLE$

\hspace{2.05cm}
$\ldots\CIRCLE \underline{\Circle} \Circle  \Circle \Circle \Circle\ldots$ \hspace{0.5cm} $\tilde{\xi}$\hspace{1.95cm} $\ldots\CIRCLE \underline\CIRCLE \Circle  \Circle \Circle \Circle\ldots$ \hspace{0.5cm}$\tilde{\xi}^{-1,0}$

\vspace{2mm}

 Let $\xi^0_{t}$ and $\xi^1_t$ be two copies of the  azrp starting from $\tilde{\xi}^{-1,0}$ and $\tilde{\xi}$, respectively. Since there is only one discrepancy between $\tilde{\xi}^{-1,0}$ and $\tilde{\xi}$, we use the basic coupling, so that particles  use the same realizations of the clocks in order to jump, see Section \ref{sec coupling and discrepancy}. Moreover, by the attractiveness property together with the conservation of the number of particles, there is at most one discrepancy between $\xi^0_{t}$ and $\xi^1_t$ at any time $t>0$.
It is simple to check that this discrepancy behaves as a second class particle.

 Now, we notice that it might happen that the discrepancy disappears at some time $t$. This corresponds to a jump of the discrepancy to the site $-1$, which has infinitely many first class particles. If this happens then the second class particle gets trapped at the site $-1$ and for this reason in partially asymmetric jumps, the limiting distribution has an atom at $u=0$ and the current diverges.

Now, let $X_{2}(t)$ denote the position at time $t$ of this discrepancy. By construction $X_{2}(0)=0$.
  We notice that, since there is only one discrepancy  between $\tilde{\xi}$ and $\tilde{\xi}^{-1,0}$, the difference between the counting processes $J_{t}^{u}(\tilde{\xi}^{-1,0})$ and $J^{u}_t(\tilde{\xi})$ in  \eqref{ec:1} is at most one and it is non-zero if and only if the discrepancy is at $ut$ or at its right hand side, at time $t$, therefore $J_{t}^{u}(\tilde{\xi}^{-1,0})-J_{t}^{u}(\tilde{\xi})=\textbf{1}{\{X_{2}(t)\geq{{ut}}\}}$.  Moreover, due to the chosen initial condition we have that $\tilde{\xi}_t(-1)=\infty$ for all time $t\geq{0}$. Then, by assumption \eqref{limit condition} we obtain that
  \begin{equation}
\frac{d}{dt}E_{\mu_{\infty,0}}[J_{t}^{u}(\tilde{\xi})]=p\,\ell\,\bar{P}(X_{2}(t)\geq{ut}),
\end{equation}
 where $\bar{P}$ denotes the coupling measure.

Now we compute $\frac{d}{dt}E_{\mu_{\infty,0}}[J_{t}^{u}(\tilde{\xi})]$ in a different way. By Dynkin's formula we decompose $J_{t}^{u}(\tilde{\xi})$ as a martingale $\mathcal{M}_{t}(\tilde{\xi})$ plus a compensator as: $$J_{t}^{u}(\tilde{\xi})=\mathcal{M}_{t}(\tilde{\xi})+\int_{0}^{t}pg( \tilde{\xi}_s(us-1))-qg(\tilde{\xi}_s(us))ds.$$
For that, it is enough to notice that the process $\{(\xi_t,J_t^u(\xi));t\geq{0}\}$ has generator $\Omega$ given  on local functions $f:\mathbb{N}_0^{\mathbb{Z}}\times\mathbb{Z}\rightarrow{\mathbb{R}}$ by $\Omega f(\xi,J)$ equal to
\begin{equation*}
\begin{split}
&pg(\xi(ut-1))[f(\xi^{ut-1,ut},J+1)-f(\xi,J)]+qg(\xi(ut))[f(\xi^{ut-1,ut},J-1)-f(\xi,J)]\\
&+\sum_{x\neq{ut-1}}pg(\xi(x))[f(\xi^{x,x+1},J)-f(\xi,J)]+\sum_{x\neq{ut}}qg(\xi(x))[f(\xi^{x,x-1},J)-f(\xi,J)]
\end{split}
\end{equation*}
and that $\Omega(J_s^u(\tilde{\xi}))=\Omega(\pi_2(\tilde{\xi}_s,J_s^u(\tilde{\xi})))$, where $\pi_2(\xi,J)=J$. Since the mean of martingales is constant and the martingale $\mathcal{M}_{t}$ vanishes at time $0$, we get that
\begin{equation*}
\frac{d}{dt}E_{\mu_{\infty,0}}[J_{t}^{u}(\tilde{\xi})]=p\,E_{\mu_{\infty,0}}[g(\tilde{\xi}_t(ut-1))]-q\,E_{\mu_{\infty,0}}[g(\tilde{\xi}_{t}(ut))].
\end{equation*}

Now, by the {\em{convergence to local equilibrium}} (see \cite{B.F.}) it follows that
\begin{equation*}
\begin{split}
&\lim_{t\rightarrow{+\infty}}E_{\mu_{\infty,0}}[g(\tilde{\xi}_{t}(ut-1))]=\tilde{g}(\rho(1,u)),\\&\lim_{t\rightarrow{+\infty}}E_{\mu_{\infty,0}}[g(\tilde{\xi}_{t}(ut))]=\tilde{g}(\rho(1,u)),
\end{split}
\end{equation*}
where $\rho(t,u)$ is given in \eqref{sol g rate} with $\rho=\infty$ and $\lambda=0$.
 Collecting these facts together we obtain that $$\lim_{t\to+\infty}p\,\ell\,\bar{P}(X_{2}(t)\geq{ut})=F(\rho(1,u)).$$
 Then, the limiting distribution is given by $$F_X(u):=1-\cfrac{F(\rho(1,u))}{p \,\ell}$$
for $u\in{[F'(+\infty),F'(0)]}$. A simple computation shows that $F'(+\infty)=0$ and $F'(0)=(p-q)$. Now we notice that $F_X(0)=q/p$ and easily one shows that $F_X(\cdot)$ is a distribution function. This finishes the proof of the first claim.

Now, we prove the L.L.N. for the current of first class particles that cross over the second class particle up to the time $t$.
In $\tilde{\xi}$ we label the first class particles at site $-1$ from the bottom to the top and we denote by $X_{1}(t)$ the position at time $t$ of the first jumping first class particle. Since first class particles preserve their order, the number of first class particles at the right hand side of the second class particle (or at its site) at time $t$, is equal to the number of first class particles  between the sites of $X_{1}(t)$ and of $X_{2}(t)$ at time $t$. Then
$$J_{2}^{zr}(t)=\sum_{x\geq{X_{2}(t)}}^{X_{1}(t)}\xi_{t}(x).$$
Notice that at positive sites, $\xi(x)$ is distributed according to $\mu_{0}$. A L.L.N. for $X_{1}(t)$ holds (see \cite{Seth.} for example): $\lim_{t\rightarrow{+\infty}}\frac{X_{1}(t)}{t}=\frac{F(\lambda)}{\lambda}$, in $\mu_0$-probability.
Since $\lambda=0$, the previous limit equals to $F'(0)$.

Now, let $f_t(at,bt)=\frac{1}{t}\sum_{x=at}^{bt} \xi_t(x)$. As a consequence of  the hydrodynamic limit derived in \cite{R.}, we have that
$$\lim_{t\rightarrow{+\infty}}\frac{1}{t}\sum_{x=at}^{bt} \xi_t(x)=\int_{a}^{b}\rho(1,u)du,$$ in $\mu_{\rho,\lambda}S_t$-probability,
where $\rho(t,u)$  is given in \eqref{sol g rate} with $\rho=\infty$ and $\lambda=0$. Since $\frac{J_{2}^{zr}(t)}{t}=f_t(X_2(t)/t,X_1(t)/t)$ and $X_2(t)/t$ (resp. $X_1(t)/t$) converges in distribution, as $t$ goes to $+\infty$, to $X$ (resp. $F'(0)$),  we get that:
\begin{equation*}
\lim_{t\rightarrow{+\infty}}\frac{J_{2}^{zr}(t)}{t}=\int_{X}^{F'(0)}\rho(1,u)du, \quad\ \textrm{in distribution}.
\end{equation*}
Now, if $p=1$, then $X$ has no atoms and it is distributed on $[0,1]$. As a consequence we get
\begin{equation*}
\lim_{t\rightarrow{+\infty}}\frac{J_{2}^{zr}(t)}{t}=\int_{X}^{F'(0)}(F')^{-1}(u)du, \quad\ \textrm{in distribution}
\end{equation*}
and a simple computation finishes the proof of the first claim.
Now, if $p\neq{1}$, then $X$ has an atom at $u=0$ and since $\rho(t,0)=+\infty$ the integral diverges.
This ends the proof.
\end{proof}

\section{Coupling with tasep}\label{sec coupling}

\subsection{Proof of Theorem \ref{convergence on the azrp}}

First of all, we recall a result that we will use in the sequel:



\begin{theorem} (\cite{F.G.M.,F.K., F.P., M.G.})\label{convergence to uniform of scp of tasep}
\\
Consider the asep starting from $\nu_{\alpha,\beta}$ with $0\leq{\beta}<{\alpha}\leq{1}$. At time zero put a second class
particle at the origin regardless the value of the configuration at this point. Let $Y_{2}(t)$ denote the position of the second class particle at time $t$. Then
$$\lim_{t\rightarrow{+\infty}}\frac{Y_{2}(t)}{t}=\mathcal{U}, \quad\ \textrm{in distribution}$$
where $\mathcal{U}$ is uniformly distributed on $[(p-q)(1-2\alpha),(p-q)(1-2\beta)]$. For $p=1$ the convergence above is almost surely.
\end{theorem}
The proof of last result for convergence in distribution and for $p=1$ was given in \cite{F.K.}, while for $p\neq1$ it was given in \cite{F.G.M.}. Under almost sure convergence and for $p=1$, last result was derived in \cite{F.P.} and \cite{M.G.}.
Now we prove Theorem \ref{convergence on the azrp}.

The idea of the proof of Theorem \ref{convergence on the azrp} consists in using the particle to particle coupling between the constant-rate tazrp and the tasep, in such a way that we are able to relate the position of the second class particle on the constant-rate tazrp with some microscopic function on the tasep. Now we
explain the relation between the particles of the two processes.

Start the constant-rate tazrp from a configuration distributed according to $\mu_{\rho,\lambda}$ and add a second class particle at the origin. On the other hand, start the tasep from $\nu_{\alpha,\beta}$, add a second class particle at the origin and remove the first class particle from there (if necessary). We notice that the origin, for both configurations, has a second class particle. Now, on the tasep, initially label the holes by denoting the position of the $i$-th hole at time $0$ by $x_{i}(0)$. To simplify notation, we label the leftmost (resp. rightmost) hole at the right (resp. left) hand side of the second class particle at time $t=0$ by $1$ (resp. $-1$). Both processes are related in such a way that basically on the tasep the distance between two consecutive holes minus one is the number of particles at a site in the tazrp, but near the second class particle one has to be more careful. We define:
 \begin{itemize}
  \item for $i=Y_2(0)-1$: $\xi(i)$ is the number of particles between $Y_2(0)$ and the first hole to its left in the tasep, therefore, $\xi(i)=Y_2(0)-x_{-1}(0)-1$;

\vspace{0.1cm}

  \item for $i=Y_2(0)$: $\xi(i)$ has a second class particle plus a number of first class particles that coincides with the number of first class particles between $Y_2(0)$ and the first hole to its right in the tasep, therefore, $\xi(i)$ has $x_{1}(0)-Y_2(0)-1$ first class particles and a second class particle;

\vspace{0.1cm}

 \item
 for $i\in{\mathbb Z\setminus{\{Y_2(0)-1,Y_2(0)\}}}$: $\xi(i)$ corresponds to the number of particles between consecutive holes, therefore, for $\kappa>0$ and for $ i=Y_2(0)+\kappa$, $\xi(i)=x_{\kappa+1}(0)-x_{\kappa}(0)-1$, similarly for $\kappa<0$;
  \end{itemize}

With the established relations we notice that for a positive site (resp. negative site) if in the  constant-rate tazrp there are $k$ particles at a given site, then for the tasep there are $k$ particles plus a hole to their right (resp. left). For positive ( resp. negative) sites there are $k$ particles at that site with probability $\alpha^k(1-\alpha)$ (resp. $\beta^k(1-\beta))$. For the constant-rate tazrp at the site $X_2(t)$ there are $k$ particles if in the tasep there are $k$ particles plus a hole to the right of the second class particle. By the definition of the invariant measures for the constant-rate tazrp we have that $\alpha=\rho/(1+\rho)$ and $\beta=\lambda/(1+\lambda)$, see Section \ref{invariant measures}.

On the figure below, we represent a possible initial configuration of the constant-rate tazrp and the corresponding configuration in the tasep. In both the tasep and tazrp the second class particle is represented by ${\circledast}$.

\vspace{0.2cm}

\hspace{3mm}
$ \CIRCLE$

\hspace{3mm}
$ \CIRCLE \CIRCLE \CIRCLE$

$\ldots \CIRCLE \CIRCLE{\underline{\circledast}}\CIRCLE \Circle \Circle \CIRCLE \Circle \ldots$ \hspace{5mm}
$\ldots \Circle \CIRCLE \CIRCLE \CIRCLE \Circle \CIRCLE \CIRCLE \underline{\circledast} \CIRCLE \Circle \CIRCLE \Circle \Circle \Circle\CIRCLE \Circle \Circle \ldots$

\vspace{0.15cm}

Now, in the presence of the particle to particle coupling we describe the relation between both dynamics. Suppose to start the constant-rate tazrp from a configuration where the second class particle stands alone at the origin and label the particles in both processes, see the figure below.

\vspace{0.2cm}

\hspace{3mm}
$ \CIRCLE^{-5}$

\hspace{3mm}
$ \CIRCLE^{-4} \CIRCLE^{-2}$

$\ldots \CIRCLE^{-3} \CIRCLE^{-1} \underline{\circledast}^0 \Circle \Circle \CIRCLE^1  \ldots$ \hspace{5mm}
$\ldots\Circle \CIRCLE^{-5} \CIRCLE^{-4} \CIRCLE^{-3} \Circle \CIRCLE^{-2} \CIRCLE^{-1} \underline{\circledast}^0 \Circle \Circle \Circle \CIRCLE^1\Circle\ldots$

\vspace{0.2cm}

Under the particle to particle coupling, particles are attached by the labels and only one configuration looks at the clocks, let us say the configuration in the zero-range process. When the clock rings at a site which is occupied only with first class particles, the jumping particle is the one with lowest label. Therefore, if in the tazrp the clock rings at a site and the $i$-th particle jumps, then the $i$-th particle in the tasep also jumps.

 Now, if the clock at the origin rings, then in the constant-rate tazrp the
second class particle jumps to the site $1$ and this corresponds in the tasep to a jump of the second class particle over the leftmost hole at its right, see the figure below.

\vspace{0.2cm}

\hspace{3mm}
$ \CIRCLE^{-5}$

\hspace{3mm}
$ \CIRCLE^{-4} \CIRCLE^{-2}$

$\ldots \CIRCLE^{-3} \CIRCLE^{-1} \underline{\Circle} {\circledast}^0 \Circle \CIRCLE^1  \ldots$ \hspace{5mm}
$\ldots \Circle \CIRCLE^{-5} \CIRCLE^{-4} \CIRCLE^{-3} \Circle \CIRCLE^{-2} \CIRCLE^{-1} \underline \Circle {\circledast}^0\Circle \Circle \CIRCLE^1\Circle\ldots$

\vspace{0.2cm}

Alternatively, if the clock at the site $-1$ rings, then the first class particle with label $-1$ jumps to the origin and in that case the second class particle does not move. In the tasep, this corresponds to a jump over the second class particle of the rightmost first class particle at its left, see the figure below.

\vspace{0.2cm}

\hspace{3mm}
$\CIRCLE^{-5}$

\hspace{3mm}
$ \CIRCLE^{-4} \hspace{6.5mm} \CIRCLE^{-1}$

$\ldots \CIRCLE^{-3} \CIRCLE^{-2} \underline{\circledast}^0 \Circle \Circle \CIRCLE^1  \ldots$ \hspace{5mm}
$\ldots\Circle \CIRCLE^{-5} \CIRCLE^{-4} \CIRCLE^{-3} \Circle \CIRCLE^{-2} {\circledast}^0 \underline\CIRCLE^{-1} \Circle \Circle \Circle \CIRCLE^1\Circle\ldots$

\vspace{0.2cm}

Now, in the zero-range the second class particle cannot jump since there is a first class particle at its site, neither the second class particle in tasep, since it is trapped between two first class particles, nevertheless the second class particle can jump to the left if the first class particle at its left decides to jump to the right. In last case, this corresponds in the tazrp to a jump of the first class particle at the site $-1$ to the origin.
Under
this relation, the position of second class particle at time $t$ in the constant-rate tazrp, corresponds to the current of holes through the second class particle up to time $t$, in the tasep, namely the number of holes that jump from $Y_2(s)+1$ to $Y_2(s)$ minus the number of holes that jump from $Y_2(s)$ to $Y_2(s)+1$ for $s\in{[0,t]}$, where we recall that for $s\geq{0}$, $Y_2(s)$ denotes the position of the second class particle in tasep at time $s$.

In order to summarize the established relations we introduce some notation. Recall that $X_{2}(t)$ denotes the position of the second class
particle at time $t$ for the constant-rate tazrp starting from $\mu_{\rho,\lambda}$ in which we add a second class particle at the origin. Let $J^{se}_{2}(t)$ (resp. $J_{2}^{zr}(t)$) denote the current of first class particles that cross over the
second class particle at time $t$ in the tasep (resp. tazrp) starting from  $\nu_{\alpha,\beta}$ with a second class particle at the origin (resp. $\mu_{\rho,\lambda}$ in which we add a second class particle at the origin). Let $H_{2}^{se}(t)$  denote the current of holes that cross over the second class particle in the tasep up to the time $t$ starting from $\nu_{\alpha,\beta}$ with a single second class particle at the origin.

From the discussion above and for these initial conditions, we have that $$J_{2}^{se}(t)=J_{2}^{zr}(t)\quad \textrm{and} \quad H_{2}^{se}(t)=X_{2}(t).$$ Now, the proof ends by recalling the following strong L.L.N. from \cite{F.M.P.}
\begin{equation*}
\lim_{t\rightarrow{+\infty}}\frac{J_{2}^{se}(t)}{t}=\Big(\frac{1-\mathcal{U}}{2}\Big)^2, \quad\ \textrm{almost  surely},
\end{equation*}
\begin{equation*} \label{zerosconvergence}
\lim_{t\rightarrow{+\infty}}\frac{H_{2}^{se}(t)}{t}=\Big(\frac{1+\mathcal{U}}{2}\Big)^2,  \quad\ \textrm{almost  surely},
\end{equation*}
where $\mathcal{U}$ is a Uniform random variable on $[(1-2\alpha),(1-2\beta)]$, where $\alpha=\rho/(1+\rho)$ and $\beta=\lambda/(1+\lambda)$. This finishes the proof.







\subsection{Crossing probabilities}\label{crossing prob}

In this section, by coupling the constant-rate tazrp with the tasep, we obtain the following result.
\begin{corollary}
Consider the constant-rate tazrp starting from $\mu_{\infty,0}$. At the initial time add a second (resp. third) class particle at the origin (resp. at site $1$) and denote by $X_{2}(t)$ ($X_{3}(t)$) its position at time $t$. Then
\begin{equation*}
\lim_{t\rightarrow{+\infty}}P(X_{2}(t)\geq{X_{3}(t)})=\frac{2}{3}.
\end{equation*}
\end{corollary}
\begin{proof}
The proof is a consequence of the coupling between the constant-rate tazrp and the tasep introduced above and of Theorem 2.3 a) in \cite{F.G.M.}.

Suppose to start the tasep from $\nu_{1,0}$ in which we add a second (resp. third) class particle at the origin (resp. at site $1$). Denote the position of the second (resp. third) class particle at time $t$ by $Y_2(t)$ (resp. $Y_3(t)$). As we have done above, on the tasep, initially we label the holes by denoting the position of the $i$-th hole at time $0$ by $x_{i}(0)$. To simplify notation, we label the leftmost (resp. rightmost) hole at the right (resp. left) hand side of the second class particle at time $t=0$ by $1$ (resp. $-1$).

As described above, up to the time that the second class particle overtakes the third class particle in tasep, the relation between the constant-rate tazrp and the tasep is defined as follows.
 \begin{itemize}
\item for $i=Y_2(0)-1$: $\xi(i)$ is the number of particles between $Y_2(t)$ and the first hole to its left in the tasep, therefore, $\xi(i)=Y_2(0)-x_{-1}(0)-1$;

\vspace{0.1cm}

  \item for $i=Y_2(0)$: $\xi(i)$ has a second class particle plus a number of first class particles that coincides with the number of first class particles between $Y_2(0)$ and $Y_3(0)$.

      \vspace{0.1cm}

 \item for $i=Y_3(0)$: $\xi(i)$ has a third class particle plus a number of first class particles that coincides with the number of first class particles between $Y_3(0)$ and the first hole to its right in the tasep, therefore, $\xi(i)$ has $x_{1}(0)-Y_3(0)-1$ first class particles and a third class particle;

     \vspace{0.1cm}

     \item  for $i\in{\mathbb Z\setminus{\{Y_2(t)-1,Y_2(t), Y_3(t)\}}}$: $\xi(i)$ corresponds to the number of particles between consecutive holes in tasep, therefore, for $\kappa>0$ and for $ i=Y_3(0)+\kappa$, $\xi(i)=x_{\kappa+1}(0)-x_{\kappa}(0)-1$, similarly for $\kappa<0$;

\end{itemize}
  On the figure below, we represent the initial configuration of the zero-range and the corresponding configuration in the tasep. Both in the exclusion and zero-range, the second and third class particles are represented by ${\circledast}$.

\vspace{0.1cm}

\hspace{0.5cm}\vdots

\hspace{0.3cm}
$\CIRCLE$

\hspace{0.3cm}
$\CIRCLE$

$\ldots\CIRCLE \underline{\circledast}\circledast  \Circle \Circle \Circle\ldots$ \hspace{2cm}
$\ldots \CIRCLE \CIRCLE \CIRCLE \CIRCLE \underline{\circledast} {\circledast}  \Circle \Circle \Circle \Circle\ldots$

\vspace{0.2cm}

 Recall that we use the particle to particle coupling, so that we have to label the particles in both processes but only one of the configurations looks at the clocks, namely the configuration in the zero-range process. Therefore we have:

 \vspace{0.1cm}

\hspace{0.5cm}\vdots

\hspace{0.3cm}
$\CIRCLE^{-3}$

\hspace{0.3cm}
$\CIRCLE^{-2}$

$\ldots\CIRCLE^{-1} \underline{\circledast}^0\circledast^1  \Circle \Circle \Circle\ldots$ \hspace{2cm}
$\ldots \CIRCLE^{-3} \CIRCLE^{-2} \CIRCLE^{-1}  \underline{\circledast}^0 {\circledast}^1  \Circle \Circle \Circle \Circle\ldots$

\vspace{0.2cm}

 Now, suppose that the clock at the site $-1$ rings. Then, in the constant-rate tazrp the first class particle with label $-1$ jumps to the origin. In the tasep, this corresponds to a jump of the rightmost first class particle over the second class particle:

 \vspace{0.1cm}

\hspace{0.5cm}\vdots

\hspace{0.3cm}
$\CIRCLE^{-4}$

\hspace{0.3cm}
$\CIRCLE^{-3}\CIRCLE^{-1}$

$\ldots\CIRCLE^{-2} \underline{\circledast}^0\circledast^1  \Circle \Circle \Circle\ldots$ \hspace{2cm}
$\ldots \CIRCLE^{-3} \CIRCLE^{-2} {\circledast}^0  \underline{\CIRCLE}^{-1} {\circledast}^1  \Circle \Circle \Circle \Circle\ldots$

\vspace{0.2cm}

 Another possible jump in the constant-rate tazrp is a jump of the third class particle to the right. This happens if the clock at the site $1$ rings. In the tasep, this corresponds to a jump of the third class particle over the leftmost hole:

 \vspace{0.1cm}

\hspace{0.5cm}\vdots

\hspace{0.3cm}
$\CIRCLE^{-3}$

\hspace{0.3cm}
$\CIRCLE^{-2}$

$\ldots\CIRCLE^{-1} \underline{\circledast}^0  \Circle \circledast^1 \Circle \Circle\ldots$ \hspace{2cm}
$\ldots \CIRCLE^{-3} \CIRCLE^{-2} \CIRCLE^{-1}  \underline{\circledast}^0  \Circle {\circledast}^1 \Circle \Circle \Circle\ldots$

\vspace{0.2cm}

  So in this case, until the time that the second class particle overtakes the third class particle, that is they share the same site in the constant-rate tazrp, the dynamics of these particles  corresponds to the dynamics of the second and third class particles in the tasep. Therefore, the claim above follows straightforwardly from Theorem 2.3 a) of \cite{F.G.M.}.

\end{proof}

\noindent{\bf Acknowledgments.}\\
I thank the warm hospitality of ``IMPA'' and ``Courant Institute of Mathematical Sciences'', where part of this work was done. I express my gratitude to FCT for the research project PTDC/MAT/109844/2009 and to the Research Centre of Mathematics of the University of Minho, for the financial support provided by "FEDER" through the "Programa Operacional Factores de Competitividade "COMPETE" and by FCT through the research project PEst-C/MAT/UI0013/2011. I thank Pablo Ferrari and Milton Jara for nice discussions on the subject. I also thank the comments of Gideon Amir and James Martin during the conference "Developments in Coupling Workshop" held in York, UK in September 2012, that contributed to an improvement of the paper.

\end{document}